\documentclass{amsart}
\usepackage{amsmath, amsfonts, amssymb,amsthm}
\usepackage{amstext}
\usepackage{mathrsfs}
\usepackage[dvipsnames]{xcolor}
\usepackage{tikz}
\usepackage[all]{xy}
\usepackage{enumerate}
\usetikzlibrary{er,positioning}

\usepackage{lineno}

\usepackage{comment}

\theoremstyle{definition}
\newtheorem{definition}{Definition}[section]

\newtheorem{theorem}[definition]{Theorem}
\newtheorem{lemma}[definition]{Lemma}
\newtheorem{proposition}[definition]{Proposition}
\newtheorem{corollary}[definition]{Corollary}

\newcommand{\sq}[1]{\ifx#1([\else\ifx#1)]%
  \else\message{invalid use of "sq"}\fi\fi}

\DeclareMathOperator{\ord}{ord}

\DeclareMathSymbol{\idot}{\mathbin}{operators}{`\.}
\allowdisplaybreaks
\newcommand{\supp}{\mathrm{Supp}}

\begin{document}
\title[A non-integrated defect relation for holomorphic maps  into algebraic varieties]{A non-integrated defect relation for holomorphic maps   into algebraic varieties}
\author{Qili Cai}
\address{Department of Mathematics\newline
\indent University of Houston\newline
\indent Houston,  TX 77204, U.S.A.} 
\email{qcai3@cougarnet.uh.edu}
\author{Min Ru}
\address{Department of Mathematics\newline
\indent University of Houston\newline
\indent Houston,  TX 77204, U.S.A.} 
\email{minru@math.uh.edu}
\author{Chin Jui Yang}
\address{Department of Mathematics\newline
	\indent University of Houston\newline
	\indent Houston,  TX 77204, U.S.A.} 
\email{cyang36@cougarnet.uh.edu}

\begin{abstract}
In 1983, relating to the study of value distribution of the Guass maps of complete minimal surfaces in ${\Bbb R}^m$, H. Fujimoto  introduced the notion of the non-integrated defect  for  holomorphic maps of an open Riemann surface into $\mathbb{P}^n(\mathbb{C})$ and obtained some results analogous to the Nevanlinna-Cartan defect relation.  This paper establishes the   non-integrated defect relation for holomorphic maps into projective varieties. 
\end{abstract}
 \thanks{2010\ {\it Mathematics Subject Classification.}
  32H30,   53A30.}  
\thanks{The second named author is supported in part by Simon Foundations grant award \#531604 and \#52160}

\baselineskip=16truept \maketitle \pagestyle{myheadings}
\markboth{}{Non-integrated defect}
\section{Introduction}
In 1983, relating to the study of value distribution of the Guass maps of complete minimal surfaces in ${\Bbb R}^m$, H. Fujimoto \cite{Fu1} introduced the notion of the non-integrated defect for  holomorphic maps of an open Riemann surface into $\mathbb{P}^n(\mathbb{C})$ and obtained some results analogous to the Nevanlinna-Cartan defect relation.

Let's recall the definition of the non-integrated defect and the main result in \cite{Fu1}.  For notations, see the books by Min  Ru (\cite{book:minru} and \cite{book:minru2}). 
 Let $M$ be an open Riemann surface.
 Let $f$ be a holomorphic  map of\ $M$\ into\  $\mathbb{P}^n(\mathbb{C})$,  $\mu_0$ be a positive integer and $D$ be a hypersurface in $\mathbb{P}^n(\mathbb{C})$ of degree $d$ with $f(M) \not\subset D$. We denote the intersection multiplicity of the image of $f$ and $D$ at $f(p)$ by\ $\nu^f(D) (p)$\ and the pull-back of the normalized Fubini-Study metric form $\Omega$ on $\mathbb{P}^n(\mathbb{C})$  by $\Omega_f$. The $non$-$integrated$ $defect \ of \ f \ with \ respect \ to \ D \ truncated \ by \ \mu_0$ is defined by 
\begin{align*}
    \delta_{\mu_0} ^f (D):=1-\inf \{ \eta  \geq 0 ~|~\eta\ \mbox{ satisfies the condition (*)}\},
\end{align*}
where the condition $(*)$ means that there exists a bounded non-negative continuous function $h$ on $M$ such that
\begin{align*}
   d\eta \Omega_{f} + dd^c\log{h^{2}}\geq  \left[\min \{\nu^f(D), \mu_0\}\right]
\end{align*}
in the sense of currents, where we mean by $[\nu]$ the (1, 1)-current associated with a divisor $\nu$.

We also recall the definition of the classical Nevanlinna defect. Let $f : B(R_0)\subset \Bbb C\to  \mathbb{P}^n(\mathbb{C}),$ $0< R_0 \leq\infty,$ be a holomorphic map. Let $D$ be a hypersurface in $ \mathbb{P}^n(\mathbb{C})$ of degree $d.$ The  {\it  Nevanlinna defect $\delta^{f,*}_{\mu_0}(D)$ of $f$ with respect to $D$ cut by $\mu_0$} is defined by
\begin{equation*}
\delta^{f,*}_{\mu_0}(D) = 1 - \limsup_{r\to R_0}{N^{(\mu_0)}_f(r,D)\over dT_f(r)}.
\end{equation*}

 Let $\omega=\frac{\sqrt{-1}}{2\pi} a dz\wedge d{\bar z}$   be the  metric form on $M$.  We define
\begin{align*}
 {\rm Ric }(\omega) =dd^c \log a,
\end{align*}
where $d= \partial+\bar{\partial}$ and $d^c= {\sqrt{-1}\over 4\pi} (\bar{\partial}-\partial)$.
  We assume the following growth condition for $f$:  there exists a nonzero bounded  continuous real-valued function $h$ on $M$ such that 
\begin{equation}\label{growthold} \rho\Omega_f + dd^c \log h^2 \geq {\rm Ric }(\omega)
\end{equation}
for some non-negative constant $\rho$. The following is the result proved by H. Fujimoto.
\begin{theorem}[Fujimoto \cite{Fu1}]\label{fujimoto thm}
Let $M$ be a complete open Riemann surface with the metric form $\omega $. Let $f: M\to \mathbb{P}^n(\mathbb{C})$ be a holomorphic map which is linearly non-degenerate (i.e., its image is not contained in any proper subspace of $\mathbb{P}^n(\mathbb{C})$). Assume that  $f$ satisfies the growth condition (\ref{growthold}). 
Then we have
 \begin{equation*}
   \sum_{j=1}^{q}\delta^f_{n}(H_j) \leq n+1+\rho n(n+1)
\end{equation*}
    for arbitrary hyperplanes $H_{1},...,H_{q}$ in $\mathbb{P}^n(\mathbb{C})$ in general position.
\end{theorem}
In this paper, we extend  the   non-integrated defect relation to holomorphic maps from a complete open Riemann surface into projective varieties. To do so, we need to extend Fujimoto's notion. 
 Let $X$ be a projective variety. For any effective Cartier divisor $D$ on $X$, we use $[D]$ to denote the line bundle associated to $D$, and $c_1([D])$ to 
the first Chern form (by fixing any Hermitian metric on $[D]$). 
\begin{definition}\label{non-integrated defect}(See \cite{CaiRu} or \cite{Yang}) Let $X$ be a projective variety, $D$ be an effective Cartier divisor on $X$ and $A$ be an ample Cartier divisor on $X$.
 Let $M$ be an open Riemann surface and let $f: M\rightarrow X$ be a holomorphic map with $f(M) \not\subset D$.   The  {\it non-integrated defect of $f$ with respect to $A$ and $D$ cut by $\mu_0$} is defined by 
$$\delta_{\mu_0}^{f, A}(D):=\sup\ \{\delta\geq 0 ~|~\delta \mbox{ satisfies the condition }~(**)\}$$
 where   $(**)$ means that there exists a bounded non-negative continuous function $h$ on $M$   such that 
\begin{align} \label{non-int}
 f^*c_1([D])-\delta f^*c_1([A]) + dd^c\log h^2 \geq \left[\min\{\nu^f(D), \mu_0\}\right].
\end{align}
\end{definition} 
If $M = B(R_0),$ $0< R_0 \leq \infty,$ is an open disc of radius $R_0$ in $\Bbb C$ and  if  $ \mathop {\lim }\limits_{r \to R_0}T_{f, A}(r)=\infty $, then (see Proposition \ref{defectcomparision} below) $$
	\delta_{\mu_0}^{f,A}(D) \leq 	\delta_{\mu_0}^{f,A, *}(D),$$ where $\delta_{\mu_0}^{f,A, *}(D)$ is the {\it Nevanlinna defect with respect to $A$ and $D$ cut by $\mu_0$} which is given by 
$$\delta_{\mu_0}^{f, A, *}(D):=\liminf_{r\rightarrow R_0} {T_{f, D}(r)-N^{(\mu_0)}_f(r, D)\over T_{f, A}(r)}.$$	

Let $D_1, D_2$ be two divisors, at least one of them is ample.  We define
\begin{equation}\label{gammacon}
\gamma_{D_1, D_2}:=\inf \{t\in {\Bbb R}~|~tc_1([D_1])+c_1([D_2])>0\}.
\end{equation}
If $D$ is ample on $X$ and $\nu^f(D)(p)\ge \mu$ for every point $p\in f^{-1}(D)$, then (for the proof, see the next section),
\begin{equation}\label{mul} \delta_{\mu_0}^{f,A}(D)\ge \left(1-{\mu_0\over \mu}\right){1\over\gamma_{D, -A}}.
\end{equation}
In particular, if the image of $f$ omits $D$, then  $\delta_{\mu_0}^{f,A}(D)\ge {1\over\gamma_{D, -A}}.$

The first result of this paper is the following theorem  which extends the result of Fujimoto stated above. 
\begin{theorem}\label{yamanoinonint} Let $M$ be a complete open Riemann surface with the  metric form $\omega$. 	
	Let $X$ be a smooth projective variety of dimension $n$ and $D=D_1+\cdots+D_q$ be a simple normal crossing divisor on $X$. Let $L$ be a line bundle on $X$ such that $L \otimes [D]$ is ample.  Let $A$ be an ample divisor on $X$.  Assume that there exists a $k$-jet differential 
	 \begin{equation*}
	\omega \in H^0( X,E_{k,m}^{GG}\Omega _X \otimes L^{-1} )
	\end{equation*}
	such that $J_kD_1+\cdots+J_kD_q \subset \pi_k^*(\omega=0)$ where  $\pi_k:J_kX \to X$ is the natural projection. Let $f: M\to X$ be a non-constant holomorphic map with $f(M) \not\subset  \text {supp}(D)$. 
 Assume the following growth condition for $f$:  there exists a nonzero bounded  continuous real-valued function $h$ on $M$ such that 
\begin{equation}\label{growthconditionruold} \rho f^*c_1([A]) + dd^c \log h^2 \geq {\rm Ric }(\omega)
\end{equation}
for some non-negative constant $\rho$.
	 If $f^*\omega \not \equiv 0$, then 
	 \begin{equation}\label{Ynonintdefect}  \sum_{j=1}^q   \delta_{k}^{f, A}(D_j)
	 \leq  \gamma_{A, L}+ 2\rho m, \end{equation}
	 where $\gamma_{A, L}$ is defined in (\ref{gammacon}).
	\end{theorem}
The above result is based on the following Second Main Theorem which is used to derive  the classical Nevanlinna defect relation. The result is originally due to 
  K. Yamanoi \cite{Yamanoi}, but here we obtain a better error term.
\begin{theorem}\label{smtyamanoi}
	Let $X$ be a smooth projective variety of dimension $n$ and $D=D_1+\cdots+D_q$ be a simple normal crossing divisor on $X$. Let $L$ be a line bundle on $X$ such that $L \otimes [D]$ is ample. Assume that there exists a  $k$-jet differential 
	 \begin{equation*}
	\omega \in H^0( X,E_{k,m}^{GG}\Omega _X \otimes L^{-1} )
	\end{equation*}
	such that $J_kD_1+\cdots+J_kD_q \subset \pi_k^*(\omega=0)$ where  $\pi_k:J_kX \to X$ is the natural projection. Let $f:B(R_0)\subset {\Bbb C} \to X,$   $0<R_0 \leq \infty$, be a non-constant holomorphic map with $f(B(R_0)) \not\subset  \text {supp}(D)$. If $f^*\omega \not \equiv 0$, then, for any $\epsilon>0,$
	 \begin{equation*}
	T_{f, L \otimes [D]}(r) \leq \sum\limits_{j= 1}^q N^{(k)}_f( r, D_j ) + S(r,f,A),
	\end{equation*}
	where $A$ is an ample divisor on $X$ and $S(r,f,A)$ is evaluated as follows:

	(i) In the case $R_0 = \infty$,  
\begin{equation*}
S(r,f,A) \leq m\Big((1+\epsilon) \log^+T_{f, A}(r)+(1+\epsilon)^2\log^+\log^+T_{f, A}(r)\Big)+O(1)
\end{equation*}
for every $r \in [0,\infty)$ excluding a set $E$ with $\int_E dt < \infty$.

	(ii) In the case $R_0 < \infty$,
		\begin{equation*}
S(r,f,A) \leq K \Big(\log{1 \over R_0-r} + \log^+T_{f, A}(r)\Big)
	\end{equation*}
for every $r \in [0,R_0)$ excluding a set $E$ with $\int_E {dt \over R_0-t} < \infty$, where $K$ is a positive constant.

	\end{theorem}
To see how Theorem \ref{yamanoinonint}   implies Fujimoto's result, as well how Theorem \ref{smtyamanoi} gives H. Cartan's theorem, 
we let $X={\Bbb P}^n(\Bbb C)$,  $D = H_1 + \cdots + H_q$, where  $H_1, \dots, H_q$ are hyperplanes in ${\Bbb P}^n(\Bbb C)$ in general position. Take
 $L=K_{\mathbb P ^n (\mathbb C)}= \mathcal O_{\mathbb P ^n (\mathbb C)} (-n-1)$, the canonical line bundle of ${\Bbb P}^n({\Bbb C})$. Then $L \otimes [D]  \cong \mathcal O_{\mathbb P ^n} (q-n-1) $. Since $q \geq n+2$, $L \otimes [D]$ is very ample. 
We construct (see the argument in section 3 below) a jet differential  $$\mathcal P\in H^0 \Big( X,E_{n,{n(n+1) \over 2}}^{GG}\Omega _{\mathbb P ^n (\mathbb C)} \otimes L^{-1}\Big)$$ 
through  the Wronskian:
	\begin{equation*}
\mathcal P = \det \left( {\begin{array}{*{20}{c}}
	z_0 &  \ldots  & z_n \\	
	dz_0 &  \ldots  & dz_n \\
	\vdots  &  \ddots  &  \vdots   \\
	{d^n z_0} &  \cdots  & {d^ nz_n}  \\
	
	\end{array} } \right).
\end{equation*}
Here we view $z_0, \dots, z_n$ as global sections of $\mathcal O_{\mathbb P ^n(\mathbb C)} (1) $.
We can verify (see the argument in section 3 below)  that $J_nH_1 + \cdots + J_nH_q \subset \pi_n^*(\mathcal P = 0)$. Take
$A ={\mathcal O}_{\mathbb P ^n(\mathbb C)} (1)$. Note that 
$\gamma_{A, L}=n+1$, and notice that $m=n(n+1)/2$,  we derive the result of H. Fujimoto stated in Theorem \ref{fujimoto thm}.

Our next result is about the non-integrated defect relation for holomorphic maps into projective spaces intersecting generic hypersurfaces of
high degree. 

\begin{theorem}\label{Brotbekdefect} (i) 
	Let $X$ be a smooth projective variety of dimension $n \geq 2$. Fix a very ample divisor $A$ on $X$. Fix a positive integer $c$. Let $D \in |dA|$ be a generic smooth hypersurface in $X$ with
\begin{equation*}
		d \geq (n+1)^{n+3} \Big( n+1+{c \over 2}\Big)^{n+3}.
	\end{equation*}
	Let $f : B(R_0) \subset {\Bbb C} \rightarrow X,$ $0 < R_0 \leq \infty,$
	 be a non-constant holomorphic map with $f(B(R_0)) \not \subset \supp D$. Assume that either 
	  $R_0 = \infty$, or $R_0<\infty$ but 
		\begin{equation*}
		\limsup_{r\rightarrow R_{0}}  {T_{f, A}(r) \over  \log {1 \over R_0 -r}} = \infty.
			\end{equation*}
Then 
			\begin{equation*}
	\delta _1^{f,A, *}(D)  \leq d - c.
	\end{equation*}

	(ii)   With the same assumptions for $X$ and $D$. Let $f :  M \rightarrow X$ 
	  	 be a non-constant holomorphic map with $f(M) \not \subset \supp D$, where $M$ is a complete open Riemann surface 
		 with the metric form $\omega$. 
	 Assume that $f$ satisfies the following growth condition: 
there exists a non-zero bounded continuous real-valued function $h$ on $M$ such that 
		\begin{equation*}
		\rho f^*c_1([A]) + dd ^c \log h^2 \geq {\rm Ric} (\omega)
	\end{equation*}
	for some non-negative constant $\rho$. Then we have 
	\begin{equation*}
	 \delta _1^{f, A} (D) \leq d- (c-2\rho).
	\end{equation*}
	\end{theorem}

We discuss some consequences of the above theorem. Let $X={\Bbb P}^n(\Bbb C)$ and $A={\mathcal O}_{\mathbb P ^n(\mathbb C)} (1)$. Let $f: {\Bbb C} \rightarrow {\Bbb P}^n(\Bbb C)$   a holomorphic map. 
Let $D$ be a hypersurface in ${\Bbb P}^n(\Bbb C)$ of degree  $d$.
Assume that the image of $f$ omits  $D$.  Then, from the definition, we can easily get $\delta _1^{f,A, *}(D)=d$. 
Hence, we  derives the following result from (i) in Theorem \ref{Brotbekdefect} by taking $c=2$:
	 {\it Let $f: {\Bbb C} \rightarrow {\Bbb P}^n(\Bbb C)$  with $n\ge 2$ be a holomorphic map. If the image of $f$  omits a generic smooth hypersurface $D$ in ${\Bbb P}^n(\Bbb C)$ of degree at least  $((n+1)(n+2))^{n+3}$, then $f$ must be constant.} 

From (\ref{mul}),  if the image of $f$ omits  $D$, then  $ \delta _1^{f, A} (D)\ge d$.
Thus, by applying  (ii) in Theorem \ref{Brotbekdefect} with $\rho=1$ and $c=2$, we can derive the following statement which is due to D.T. Nuynh (see \cite{H}):  {\it Let $M$  be a complete minimal surface immersed in 
${\Bbb R}^n$ and let $G : M\rightarrow  {\Bbb P}^{n-1}(\Bbb C)$  be its generalized Gauss map. If $G$ omits a generic smooth hypersurface $D$ in ${\Bbb P}^{n-1}(\Bbb C)$ of degree at least  $(n(n+1))^{n+2},$ then $M$ must be flat.}
 For details, see \cite{Fu1} or \cite{H}.

Theorem \ref{Brotbekdefect} is  based on the following result   which  improves the result of 
   D. T. Huynh, D. V. Vu and S. Y.  Xie (see \cite{HVX}, Theorem 3.1) with a better error term.
  \begin{theorem}\label{smtXie}
	Let $X$ be a smooth projective variety of dimension $n$.
	Let $D$ be a normal crossing divisor on $X$ and $A$ be an ample line bundle on $X$.  Let  
	\begin{equation*}
	\mathcal P \in H^0( X,E_{k,m}^{GG}\Omega _X( \log D ) \otimes  A^{-\tilde m}  ).
	\end{equation*}
	Let $f: B(R_0)\subset {\Bbb C} \rightarrow X$, $0<R_0 \leq \infty$, be a holomorphic map with $f(B(R_0)) \not\subset \text {supp}(D)$ and $\mathcal P (j_k(f)) \not  \equiv 0$.  Then, for any $\epsilon > 0,$ 
	\begin{equation*}
	\tilde mT_{f, A}(r) \leq mN^{(1)}_f(r, D)+S(r,f,A),
	\end{equation*}
	where $S(r,f,A)$ is evaluated as in Theorem \ref{smtyamanoi}.
	\end{theorem}

\section{Notations and preparations}
In this section, we prove a Main Lemma for jet differentials which will be used in the proofs of non-integrated defect relations.

We first recall a lemma regarding the logarithmic derivatives. 
\begin{lemma}[\cite{book:minru2}, Proposition 4.2.1]\label{fujimoto}
Let $\varphi$ be a nonzero meromorphic function on $B(R_0)\subset {\Bbb C},$ $0<R_0 \leq \infty$. Let  $l$ be a  positive integer and $t, p,  r_0$ be real numbers with $0<tl<p<1$ and $0<r_0<R_0$. Then there exists a positive constant $K$ such that, for $r_0<r<R<R_0$, 
 \begin{equation*}
   {\int_0^{2\pi } {\left| \Big( {\varphi ^{(l)} \over
   				\varphi } \Big)(re^{i\theta } )\right|}^t} {d\theta \over
   2\pi } \leq K \Big( {R  \over r(R-r)}  T_{\varphi}(R)\Big)^{p}.
\end{equation*}
\end{lemma}
Note that Proposition 4.2.1 in   \cite{book:minru2} states slightly different from the above statement. But with some minor  modifications, as well as some minor modifications of 
Proposition 4.1.6 in  \cite{book:minru2},  it is not difficult to derive the above Lemma. For an alternative proof, see Proposition A5.1.4 in \cite{book:minru} (note: in Proposition A5.1.4, the condition $0<\alpha l<1/2$ can be weaken to  $0<\alpha l<1$).

We now recall the concept of the Green-Griffiths jet differentials and the  Green-Griffiths logarithmic  jet differentials. 

\noindent$\bullet$ {\bf Jet bundle}.   Let $X$ be a complex manifold of dimension $n$.
Let $x\in X$ and let   $J_k(X)_x$
 be  the set of equivalence classes of holomorphic maps  $f:  (\bigtriangleup, 0)\rightarrow (X, x)$, 
where $ \bigtriangleup$ is a disc of unspecified positive radius,
with the equivalence relation $f\sim g$  if and only if 
$f^{(j)}(0)=g^{(j)}(0)$  for $0\leq j\leq k$, when computed in some local coordinate
system of $X$ near $x$. The equivalence class of $f$ is denoted by $j_k(f)$, which is called the $k$-jet of $f$. A $k$-jet $j_k(f)$ is said
to be {\it regular} if $f'(0)\not=0$. 
Set 
$$J_k(X)=\cup_{x\in X} J_k(X)_x$$
and consider the natural projection
$$\pi_k: J_k(X)\rightarrow X.$$
Then $J_k(X)$ is a complex manifold which carries the structure of a holomorphic fiber bundle over $X$, which is called the {\it $k$-jet bundle over $X$}. When $k = 1,$ $J_1(X)$ is canonically isomorphic to the holomorphic tangent bundle  $TX$ of $X$.
On $J_kX$, there is a natural ${\Bbb C}^*$ action defined by, for any $\lambda\in {\Bbb C}^*$ and $j_k(f)\in J_kX$, set 
$$\lambda\cdot j_k(f)=j_k(f_{\lambda})$$
where $f_{\lambda}$ is given by  $t\mapsto f(\lambda t)$. 

\noindent$\bullet$ {\bf Jet Differential}. A {\it jet differential of order $k$} is a holomorphic map $\omega: J_kX\rightarrow {\Bbb C}$, and 
a {\it jet differential of oder $k$ and degree $m$} is a holomorphic map $\omega: J_kX\rightarrow {\Bbb C}$
such that $$\omega(\lambda\cdot j_k(f))=\lambda^m\omega(j_k(f)).$$
The  Green-Griffiths  sheaf ${\mathcal E}_{k, m}^{GG}$ of oder $k$ and degree $m$
is defined as follows: for any open set $U\subset X$, 
$${\mathcal E}_{k, m}^{GG}(U)=\{\mbox{jet differentials of order $k$ and degree $m$ on}~U\}.$$
It is a locally free sheaf, we denote its vector bundle on $X$ as $E_{k, m}^{GG}\Omega_X$.
Let $\omega \in {\mathcal E}_{k, m}^{GG}(U)$,  the differentiation 
$d\omega\in {\mathcal E}_{k+1, m+1}^{GG}(U)$ is defined by 
$$d\omega(j_{k+1}(f))= {d\over dt} \omega(j_k(f)(t))\Big|_{t=0}.$$
Note that, formally, we have $d(d^pz_j)=d^{p+1}z_j$ for the local coordinates $(z_1, \dots, z_n)$.

\noindent $\bullet$  {\bf  Logarithmic Jet Differential}.  In the logarithmic setting,  let $D \subset X$ be a normal crossing divisor on $X$. Recall that $D \subset X$ is {\it in normal crossing} 
if, at each $x\in X$,  there exist some local coordinates $z_1,...,z_r, z_{r+1},..., z_n$  centered at $x$ such that $D$ is locally defined by
$$D = \{z_1\cdots z_r = 0\}.$$ 
We define the {\it sheaf of logarithmic Green-Griffiths jet differentials of order $k$ and degree $m$} as follows:
 for any open set $U\subset X$, 
\begin{eqnarray*}
&~&{\mathcal E}_{k, m}^{GG}\Omega_X(\log D)(U)\\
&=&\{\mbox{jet differentials of order $k$ and degree $m$ on}~U\backslash D ~\mbox{with at worst log poles on}~D\},
\end{eqnarray*}
where $\omega$ has log poles on $D$ means that for every local coordinates $(z_1,\dots, z_n)$ such that 
$D = \{z_1\cdots z_r = 0\}$, we have, outside $D$,
$$\omega(z)=\sum_{|I_1|+\cdots+ k|I_k|=m} a_{I_1,\dots, I_k}(z) (dz)^{I_1}\cdots (d^kz)^{I_k}$$
where $I_p=(i_{p, 1}, \dots, i_{p,n})\in {\Bbb N}^n$, $|I_p|=i_{p, 1}+ \cdots+i_{p,n}$,
$(d^pz)^{I_p} = (d^pz_1)^{i_{p,1}}\cdots(d^pz_n)^{i_{p,n}},$
$$a_{I_1,\dots, I_k}(z)={b_{I_1,\dots, I_k}(z)\over z_1^{i_{1,1}+\cdots +i_{k, 1}}\cdots z_r^{i_{1,r}+\cdots +i_{k, r}}}$$
and $b_{I_1,\dots, I_k}$ is a holomorphic function on $U$.

We prove the following Main Lemma for jet differentials.
\begin{lemma}[Main Lemma for jet differentials]\label{jetdefest}

	(i)  Let $X$ be a smooth projective variety of dimension $n$.
Let $D$ be a normal crossing divisor on $X$.  Let  
 \begin{equation*}
\mathcal P \in H^0( X,E_{k,m}^{GG}\Omega _X( \log D ))
\end{equation*}
be a logarithmic $k$-jet ($k\geq 1$) differential. Let $f: B(R_0)\rightarrow X$, $0<R_0 \leq \infty$, be a holomorphic map such that $f(B(R_0)) \not\subset  \text {supp}(D)$ and $\mathcal P (j_k(f)) \not  \equiv 0$. Let $t, p, r_0$ be positive real numbers with $0<tm<p<1$ and $0<r_0<R_0$. Then there exists a positive constant $K$  such that, for $r_0<r<R<R_0$, 
 \begin{equation*}
\int_0^{2\pi } \big| \mathcal P( j_k( f ) )( re^{i\theta } ) \big|^t {d\theta\over 2\pi} \leq K \Big( {R\over r(R-r)} T_{f, A}(R)\Big)^p,
\end{equation*}
where $A$ is any (fixed) ample divisor on $X$.

 (ii) With the same assumptions above. Let
 $$\mathcal P \in H^0( X,E_{k,m}^{GG}\Omega _X( \log D )\otimes L^{-1}),$$
 where  $(L, h)$ is an Hermitian line bundle on $X$.
 Assume that  $\mathcal P (j_k(f)) \not  \equiv 0$. 
Then  \begin{equation*}
\int_0^{2\pi } \big\| \mathcal P( j_k( f ) )( re^{i\theta})\big\|_{h^{-1}}^t  {d\theta\over 2\pi}\leq K \Big( {R\over r(R-r)} T_{f, A}(R)\Big)^p.
\end{equation*} 
 \end{lemma}
 \begin{proof}
(i) For every point $x$ in $X$, there exists a local chart $\left(U_x, \phi_x = (z_{1,x},...,z_{n,x})\right)$ around $x$ such that $\phi_x$ is a rational function on $X$ 
and $D|_{U_x} = \{z_{1,x}\cdots z_{s,x}=0 \} $ for some $0 \leq s \leq n$. If $s=0$, we mean $D\cap{U_x} = \emptyset$. Since $X$ is compact, we can cover $X$ by finitely many such local charts $\{(U_\lambda, \phi _\lambda)\}$. Let $\{\chi_\lambda\}$ be a partition of unity subordinated to $\{U_\lambda\}$. Note that $ \text {supp}(\chi_\lambda) \subset U_\lambda$.
 Write $\phi_\lambda (f) = (f_{1,\lambda},...,f_{n,\lambda}) : f^{-1}(U_\lambda)  \to \mathbb {C}^n$. On $f^{-1}(U_\lambda)$, $\mathcal P(j_k(f))$ can be expressed as 
\begin{eqnarray}\label{locrep}
&~&\mathcal P(j_k(f))\Big|_{ f^{-1}(U_\lambda) }\\
 &=&\sum_{
 	\alpha _1,...,\alpha _k \in \mathbb N^n, \sum_{j=1}^k j|\alpha_j| =m} 
	 a_{\alpha _1,...,\alpha _k} \left( \prod\limits_{i = 1}^s \left( {f_{i,\lambda}^{(1)} \over f_{i,\lambda}}\right)^{\alpha _{1,i}}
 \prod\limits_{i = s + 1}^n (f_{i,\lambda}^{(1)})^{\alpha _{1,i}}  \right) \nonumber  \\ 
 &~&\cdots \left( \prod\limits_{i = 1}^s \left( {f_{i,\lambda}^{(k)} \over f_{i,\lambda}}\right)^{\alpha _{k,i}}\prod\limits_{i = s + 1}^n (f_{i,\lambda}^{(k)})^{\alpha _{k,i}}  \right)\nonumber
	 \end{eqnarray}
where $a_{\alpha _1,...,\alpha _k}$ are locally defined holomorphic functions on $f^{-1}(U_\lambda)$,  $\alpha_j = (\alpha_{j,1},...,\alpha_{j,n})$ and 
$|\alpha_j| = \alpha_{j,1}+\cdots +\alpha_{j,n}$ for $1 \leq j \leq k$.
Then, on $f^{-1}(U_\lambda)$,
 	 	 \begin{eqnarray*}
 		\big|\mathcal P(j_k(f))\big|&\leq&\sum\limits_{\begin{subarray}{l} 
 				\alpha _1,...,\alpha _k 
 		\end{subarray}}
 	\left|a_{\alpha _1,...,\alpha _k}\left(\prod\limits_{j = 1}^k
 	\left( \prod\limits_{i = 1}^s \left({f_{i,\lambda}^{(j)} \over f_{i,\lambda}}\right)^{\alpha _{j,i}} \prod\limits_{i = s + 1}^n (f_{i,\lambda}^{(j)})^{\alpha _{j,i}}  \right)\right) \right| \\
 		&\leq&\sum\limits_{\begin{subarray}{l} 
 				\alpha _1,...,\alpha _k 
 		\end{subarray}} \left|a_{\alpha _1,...,\alpha _k}\left( \prod\limits_{j = 1}^k 
 	\left( \prod\limits_{i = 1}^n \left({f_{i,\lambda}^{(j)} \over f_{i,\lambda}}\right)^{\alpha _{j,i}} \prod\limits_{i = s + 1}^n (f_{i,\lambda})^{\alpha _{j,i}} \right)\right)\right|,
 	\end{eqnarray*}
here $\alpha _1,...,\alpha _k$ run over $\mathbb N^n$ with $|\alpha _1| + 2|\alpha _2| + ... + k|\alpha _k| = m.$
Since $\chi_\lambda$ has compact support, we can find that 
 $$
 \left|\chi_\lambda (f)a_{\alpha _1,...,\alpha _k}\prod_{j=1}^k \left(\prod_{i=s+1}^n(f_{i,\lambda})^{\alpha_{j,i}}\right) \right|
 $$ 
 is bounded. Also note that $0<t<1$ by assumption. Hence we get
 	
 	\begin{eqnarray*}
 		\chi_\lambda (f) \big|\mathcal P(j_k(f)) \big|^t \leq C_\lambda \sum_{\alpha _1,...,\alpha _k \in \mathbb N^n} \left|\prod\limits_{j = 1}^k \prod\limits_{i = 1}^n \left({f_{i,\lambda}^{(j)} \over f_{i,\lambda}}\right)^{\alpha _{j,i}} \right|^t
 	\end{eqnarray*}
 	for some positive constant $C_\lambda$.
 Thus, we have
 \begin{eqnarray*}
 \big|\mathcal P(j_k(f)) \big|^t \leq
 \sum_\lambda C_\lambda \sum_{
 				\alpha _1,...,\alpha _k \in \mathbb N^n} \left|\prod\limits_{j = 1}^k \prod\limits_{i = 1}^n \left({f_{i,\lambda}^{(j)} \over f_{i,\lambda}}\right)^{\alpha _{j,i}} \right|^t.
 \end{eqnarray*}
Therefore, we have
\begin{eqnarray}\label{rmp}
&~& \int_0^{2\pi } \big| \mathcal P( j_k( f ) )( re^{i\theta } ) \big|^t {d\theta\over2\pi}\\
&\leq& \sum_\lambda C_\lambda \sum_{
 				\alpha _1,...,\alpha _k \in \mathbb N^n} \int_0^{2\pi }  \left|\prod\limits_{j = 1}^k \prod\limits_{i = 1}^n \left({f_{i,\lambda}^{(j)} \over f_{i,\lambda}}\right)^{\alpha _{j,i}} (re^{i\theta})\right|^t  {d\theta\over2\pi},\nonumber
\end{eqnarray}
where  $\alpha _1,...,\alpha _k$ run over $\mathbb N^n$ with $|\alpha _1| + 2|\alpha _2| + ... + k|\alpha _k| = m.$
 We now fix $\alpha _1,...,\alpha _k\in\mathbb N^n$ with $|\alpha _1| + 2|\alpha _2| + \cdots  + k|\alpha _k| = m.$
Put $s_{j,i} = j\alpha_{j,i}/m,$ then $\sum_{j=1}^{k}\sum_{i=1}^{n}s_{j,i} = 1.$ By H\"older's inequality,
\begin{eqnarray*}
&~&\int_0^{2\pi }  \left|\prod\limits_{j = 1}^k \prod\limits_{i = 1}^n \left({f_{i,\lambda}^{(j)} \over f_{i,\lambda}}\right)^{\alpha _{j,i}} (re^{i\theta})\right|^t  {d\theta\over2\pi}\\
&\leq& \prod\limits_{j = 1}^k \prod\limits_{i = 1}^n \left(\int_0^{2\pi }\left| \left({f_{i,\lambda}^{(j)} \over f_{i,\lambda}}\right)^{\alpha _{j,i}}(re^{i\theta}) \right|^{t/ s_{j,i}}  {d\theta\over2\pi}\right)^{s_{j,i}}.
\end{eqnarray*}
Note that
$$
j\cdot \alpha _{j,i} \cdot (t/ s_{j,i}) = j\cdot \alpha _{j,i} \cdot (tm/j \alpha _{j,i}) = tm < 1,
$$
so  Lemma \ref{fujimoto} implies that
\begin{eqnarray*}
&~&\left(\int_0^{2\pi }\left| \left({f_{i,\lambda}^{(j)} \over f_{i,\lambda}}\right)^{\alpha _{j,i}}(re^{i\theta}) \right|^{t/ s_{j,i}}  {d\theta\over2\pi}\right)^{s_{j,i}} \\
&\leq& \left(C_1 \Big( 
{R \over r(R-r)} T_{f_{i, \lambda}}(R)\Big)^{p}\right)^{s_{j,i}}.
\end{eqnarray*}
Since $\phi_\lambda$ is a rational function and then we can apply Theorem A2.3.5 in \cite{book:minru} to get $T_{f_{i,\lambda}}(\rho) \leq C_2 T_{f,A}(\rho)$. Hence
\begin{eqnarray*}
\int_0^{2\pi }  \left|\prod\limits_{j = 1}^k \prod\limits_{i = 1}^n \left({f_{i,\lambda}^{(j)} \over f_{i,\lambda}}\right)^{\alpha _{j,i}} (re^{i\theta})\right|^t  {d\theta\over2\pi}
&\leq& \left(C_3 \Big( 
{R \over r(R-r)} T_{f, A}(R)\Big)^{p}\right)^{\sum_{j=1}^{k}\sum_{i=1}^{n} s_{j,i}}\\
&=& C_3\Big( 
{R \over r(R-r)} T_{f, A}(R)\Big)^{p}.
\end{eqnarray*}
Combining the above with (\ref{rmp}) completes our proof.

(ii) We use the same  notations in (i), but we further assume that $L^{-1}|_{U_\lambda}$ is trivial. For each $\lambda$, let $s_\lambda$ be a nowhere vanishing local frame on $U_\lambda$ to $L^{-1}.$ Then, on $U_\lambda,$ 
\begin{equation*}
\mathcal P\Big|_{U_\lambda} = \tilde{\mathcal P}_\lambda \otimes s_\lambda
\end{equation*}
for some $\tilde{\mathcal P}_\lambda \in H^0\left( U_\lambda,E_{k,m}^{GG}\Omega _X( \log D )\right).$ On $f^{-1}(U_\lambda)$, $\tilde{\mathcal P}(j_k(f))$ can be expressed as 
\begin{eqnarray*}
    \tilde{\mathcal P}(j_k(f))\Big|_{f^{-1}(U_\lambda)}=	\sum_{\alpha _1,...,\alpha _k \in \mathbb N^n} a_{\alpha _1,...,\alpha _k} \left( \prod\limits_{i = 1}^s \left( {f_{i,\lambda}^{(1)} \over f_{i,\lambda}}\right)^{\alpha _{1,i}}
 \prod\limits_{i = s + 1}^n (f_{i,\lambda}^{(1)})^{\alpha _{1,i}}  \right) \cdots\\ \left( \prod\limits_{i = 1}^s \left( {f_{i,\lambda}^{(k)} \over f_{i,\lambda}}\right)^{\alpha _{k,i}}\prod\limits_{i = s + 1}^n (f_{i,\lambda}^{(k)})^{\alpha _{k,i}}  \right), \nonumber
	 \end{eqnarray*}
	 where  $\alpha _1,...,\alpha _k$ run over $\mathbb N^n$ with $|\alpha _1| + 2|\alpha _2| + ... + k|\alpha _k| = m.$
Therefore \begin{eqnarray*}
&~&\chi_\lambda(f)\big\|\mathcal P(j_k(f))\big\|_{h^{-1}}^t \\
&=& \chi_\lambda(f) \big|\tilde{\mathcal P}(j_k(f))\big|^t\cdot \big\|s_\lambda(f)\big\|_{h^{-1}}^t\\
&=& \sum\limits_{\begin{subarray}{l} 
 				\alpha _1,...,\alpha _k 
 		\end{subarray}}\chi_\lambda(f) \left|a_{\alpha _1,...,\alpha _k}\left( \prod\limits_{j = 1}^k 
 	\left( \prod\limits_{i = 1}^n \left({f_{i,\lambda}^{(j)} \over f_{i,\lambda}}\right)^{\alpha _{j,i}} \prod\limits_{i = s + 1}^n (f_{i,\lambda})^{\alpha _{j,i}} \right)\right)\right|^t\cdot \big\|s_\lambda(f)\big\|_{h^{-1}}^t.
\end{eqnarray*}
Since $\chi_\lambda$ has compact support, 
$$
\left|\chi_\lambda(f)a_{\alpha _1,...,\alpha _k}\prod_{j=1}^k \left(\prod_{i=s+1}^n(f_{i,\lambda})^{\alpha_{j,i}}\right)\right|\cdot \big\|s_\lambda(f)\big\|_{h^{-1}}
$$
is bounded. Hence 
$$
\big\|\mathcal P(j_k(f))\big\|_{h^{-1}}^t  \leq \sum_\lambda C_\lambda \sum_{
 				\alpha _1,...,\alpha _k \in \mathbb N^n} \left|\prod\limits_{j = 1}^k \prod\limits_{i = 1}^n \left({f_{i,\lambda}^{(j)} \over f_{i,\lambda}}\right)^{\alpha _{j,i}} \right|^t,
$$
where  $\alpha _1,...,\alpha _k$ run over $\mathbb N^n$ with $|\alpha _1| + 2|\alpha _2| + ... + k|\alpha _k| = m.$
Repeat the argument as in (i) starting from (\ref{rmp}), we obtain the conclusion.
\end{proof}

In the rest of this section we  discuss some properties about the non-integrated defect. 

\begin{proposition}\label{defectcomparision}
	Let $X$ be a smooth projective variety. Let $D$ be an effective Cartier divisor on $X$ and $A$ be an ample divisor on $X$. Let $f : B(R_0)\subset {\Bbb C} \to X,$ $0<R_0\leq\infty,$ be a holomorphic  map. Fix a positive integer $\mu_0.$ If  $ \mathop {\lim }\limits_{r \to R_0}T_{f, A}(r)=\infty $, then we have 
	\begin{equation*}
	\delta_{\mu_0}^{f,A}(D) \leq 	\delta_{\mu_0}^{f,A, *}(D).
	\end{equation*}
	\end{proposition}

\begin{proof} 
	Take a positive integer $\ell$ such that $\ell A$ is very ample. Let $\phi_A : X \to \mathbb P ^N(\Bbb C)$ be the canonical embedding associated to the linear system $|\ell A|$. Then we have $dd^c\log \big\|\phi_A(f)\big\|^2=\ell f^*c_1([A])$. Let $\psi$ be a holomorphic function on $B(R_0)$ such that $\nu_\psi = \nu ^f (D) $. Let $s_D$ be the canonical section of the line bundle $[D]$ associated to $D$. Applying the Poincar\'e-Lelong formula  (see Theorem A2.2.2 in \cite{book:minru})  to $\big\|s_D(f)\big\|_D$, where $\big\|\cdot\big\|_D$ is an Hermitian metric on $[D]$, we get $dd^c\log \big\|s_D(f)\big\|_D^2=dd^c\log \big|\psi\big|^2 - f^*c_1([D]).$ Thus
	\begin{equation}\label{Chern}
	f^*c_1([D]) = dd^c\log {\big|\psi\big|^2 \over \big\|s_D(f) \big\|_D^2}.\end{equation}
	Take arbitrary $\delta \geq 0$ satisfying the condition (**) in the definition of $\delta_{\mu_0}^{f,A}(D)$:
$$
 f^*c_1([D])-\delta f^*c_1([A]) + dd^c\log h^2 \geq \left[\min\{\nu^f(D), \mu_0\}\right],
$$
where    $h$ is a bounded non-negative continuous function on $B(R_0).$ Let $\varphi $ be a holomorphic function on $B(R_0)$ such that $\nu_\varphi =\text {min}\{\nu ^f (D),\mu_0\}$.
 Then
		\begin{equation*}
	\nu:=\log  {\big|\psi\big| \over \big\|s_D(f) \big\|_D} - {\delta \over \ell}\log \big\|\phi_A(f)\big\| +\log h -\log \big|\varphi\big|
	\end{equation*}	
	is subharmonic. Using Green-Jensen formula (see Theorem A2.2.3 in \cite{book:minru}), we have  
\begin{eqnarray*}
0 &\leq& \int_0^{2\pi} \nu(re^{i\theta}) {d\theta\over 2\pi} - \int_0^{2\pi} \nu(r_0e^{i\theta}) {d\theta\over 2\pi} \\ 
&=& 	\int_0^{2\pi} \log {\big|\psi(r e^{i\theta})\big| \over \big\|s_D(f)(re^{i\theta})\big\|_D}{d\theta\over2\pi} - {\delta \over \ell}	\int_0^{2\pi} \log  \big\|\phi_A(f)(re^{i\theta})\big\| {d\theta\over2\pi}\\
&~& + \int_0^{2\pi} \log h(re^{i\theta}) {d\theta\over2\pi} - \int_0^{2\pi} \log \big|\varphi(re^{i\theta})\big| {d\theta\over2\pi} + O(1)\\
&\leq& \int_{{r_0}}^r {dt \over t} \int_{B(t)} dd^c \log {\big|\psi\big|^2 \over \big\|s_D(f)\big\|_D^2} - {\delta \over \ell} \int_{{r_0}}^r {dt \over t} \int_{B(t)} dd^c \log \big\|\phi_A(f)\big\|^2\\
&~& - \int_0^{2\pi} \log \big|\varphi(re^{i\theta})\big| {d\theta\over2\pi} + O(1)\\
&=&  \int_{{r_0}}^r {dt \over t} \int_{B(t)} dd^cf^*c_1([D])  - \delta \int_{{r_0}}^r {dt \over t} \int_{B(t)} dd^cf^*c_1([A]) \\
&~& - N^{(\mu_0)}_f(r,D) + O(1)\\
&=& T_{f, D}(r) - \delta T_{f, A}(r)- N^{(\mu_0)}_f(r, D) + O(1).
\end{eqnarray*}
	This implies 
		\begin{equation*}
\delta \leq  {T_{f, D}(r) - N^{(\mu_0)}_f(r, D) \over T_{f, A}(r)} + {K \over T_{f, A}(r)}
	\end{equation*}	
	for some constant $K$.
	Since $\delta $ is arbitrary and $T_{f, A}(r) \to \infty$ as $r \to R_0$, we obtain our conclusion.
\end{proof}
We also  give a proof of (\ref{mul}) here. 
\begin{proof}
Let $D$ and $A$ be ample divisors on $X.$ Suppose that $\nu^f(D)(p)\ge \mu$ for every point $p\in f^{-1}(D).$ Let $\delta = (1-{\mu_0\over \mu}){1\over\gamma_{D,-A}}$ and 
$h = \big\|s_D(f)\big\|_D^{\mu_0\over \mu}.$ Note that, by the definition of $\gamma_{D,-A},$
$$
\gamma_{D,-A}c_1([D]) \geq c_1([A]).
$$
Hence, by also using  (\ref{Chern}), 
\begin{eqnarray*}
&~&f^*c_1([D])-\delta f^*c_1([A]) + dd^c\log h^2\\
&\geq& f^*c_1([D]) - \left(1-{\mu_0\over \mu}\right) {1\over\gamma_{D,-A}} \cdot  \gamma_{D,-A}f^*c_1([D]) + dd^c\log h^2\\
&=& \left({\mu_0\over \mu}\right)   f^*c_1([D]) +dd^c\log h^2
 = dd^c \log {\big|\psi\big|^{2\mu_0\over \mu}\over \big\|s_D(f)\big\|_D^{2\mu_0\over \mu}} + dd^c \log \big\|s_D(f)\big\|_D^{2\mu_0\over \mu}\\
&=& dd^c \log \big|\psi\big|^{2\mu_0\over \mu} 
= {\mu_0\over \mu} \nu^f(D)\geq \mu_0.
\end{eqnarray*}
\end{proof}

\section{Proofs of Theorem  \ref{smtXie}, Theorem  \ref{smtyamanoi} and Theorem \ref{non-integrated defect}}
\noindent{\it Proof of Theorem  \ref{smtXie}}.
\begin{proof}
  Let $(V, \phi)$ be a small enough local chart of $X$ such that $\phi: X\rightarrow {\Bbb C}^n$ is a rational map 
 and $D|_V=\{z_1 \cdots  z_s=0 \} $ for some $0 \leq s \leq n$. If $s=0$, we mean $D\cap V = \emptyset$, where $(z_{1},...,z_{n})$ is the standard  coordinate on ${\Bbb C}^n$. Put 
	$f_{j}:=z_{j}(\phi(f))$ for $j=1, \dots, n$. 
	From (\ref{locrep}), it is easy to see that the support of $\nu^\infty_{\mathcal P(j_k(f))}$ on $f^{-1}(V)$ is a subset of the zero set of $f_1,...,f_s$.
	Since, for $1 \leq j \leq k$, the pole order of $\Big({f_i^{(j)} \over f_i}\Big)^{\alpha _{j,i}}$ at any point $z$ in $f^{-1}(V)$ is at most $j \cdot \alpha_{j,i} \cdot \text {min}\{\text {ord}_zf_i,1\}$. Hence
		\begin{eqnarray}\label{pole}
	\nu^\infty_{\mathcal P(j_k(f))}(z) &\leq& \sum_{1 \leq i \leq s, 
	1 \leq j \leq k} j \cdot \alpha_{j,i}\cdot \text {min}\{\text {ord}_zf_i,1\} \\
&\leq& (|\alpha_1|+2|\alpha_2|+\cdots +k|\alpha_k|) \cdot \text{min} \{ \nu ^f(D)(z), 1\}\nonumber \\
&=& m\cdot\text {min} \{ \nu ^f(D)(z), 1\}. \nonumber\end{eqnarray}
On the other hand, by applying (ii) in Lemma  \ref{jetdefest}, we have, with $0<mt<p<1$,
\begin{equation}\label{p1}
\int_0^{2\pi } \big\|\mathcal P(j_k(f))( re^{i\theta } ) \big\|_{h^{-1}}^t {d\theta \over 2\pi} 
\leq K \Big( {R\over r(R-r)} T_{f, A}(R)\Big)^p
\end{equation}
for $0<r_0<r<R<R_0$, 
where $h$ is an Hermitian metric on $A^{\tilde m}$.
By the Poincar\'e-Lelong formula (see Theorem A2.2.2 in \cite{book:minru}), we have, by also using (\ref{pole}), 
	 $$
	\text {dd}^c \log\big\| \mathcal P( j_k( f ) ) \big\|_{h^{-1}}^2 
\geq \tilde m f^*c_1(A)-\nu^\infty_{\mathcal P(j_k(f))} \ge  \tilde m f^*c_1(A)- m\cdot \text {min} \{ \nu ^f(D), 1\}$$
	in the sense of currents. Hence,  by applying Green-Jensen's formula  (see Theorem A2.2.3 in \cite{book:minru}), we get 
	\begin{equation*}
		\tilde mT_{f, A}(r) \leq \int_0^{2\pi } \log\big\| \mathcal P(j_k(f)) \big\|_{h^{-1}} {d\theta \over 2\pi} + mN^{(1)}_f(r, D)+ O(1).
	\end{equation*}
	For any $\epsilon>0$, take $t, p$ with $0<mt<p<1$ and $p< (1+\epsilon) mt$. Then, for $r_0<r<R<R_0$, 
	by the concavity of logarithmic function and (\ref{p1}), we have 
		\begin{eqnarray}\label{log}
		~&~&\int_0^{2\pi } \log\big\| \mathcal P(j_k(f))( re^{i\theta } ) \big\|_{h^{-1}} {d\theta \over 2\pi}\\ 
&\leq& {1\over t} 
 \log^+ \int_0^{2\pi } \big\| \mathcal P(j_k(f))( re^{i\theta } ) \big\|_{h^{-1}}^t {d\theta \over 2\pi}\nonumber \\ 
&\leq&{p\over t}  \left(\log^+ T_{f, A}(R) + \log^+{R \over r(R-r)}+O(1)\right) \nonumber \\
			&\leq&m(1+\epsilon)\left(\log^+ T_{f, A}(R) + \log^+{R\over r(R-r)}\right)+O(1). \nonumber
	\end{eqnarray}
We now use  Lemma 4.1.7 in \cite{book:minru} to  estimate the above term. Denote the above term by $S(r, f, A)$.  In the case when $R_0=\infty$, 
take $R=  r + {1 \over \log^{1+\epsilon} T_{f,A}(r)}$.
 Lemma 4.1.7  in \cite{book:minru} implies that 
$$\log T_{f, A}(R) \leq  \log T_{f, A} (r) +1  $$ and
$$
\log^+{R\over r(R-r)}\leq (1+\epsilon)\log^+\log T_{f,A}(r) + O(1)
$$
  for every $r \in [0,\infty)$ excluding a set $E$ with $\int_E dt < \infty$.
Thus 
\begin{equation*}
S(r,f,A) \leq  m\Big((1+\epsilon)\log^+T_{f, A}(r)+(1+\epsilon)^2\log^+\log^+T_{f, A}(r) \Big)+O(1)
\end{equation*}
for every $r \in [0,\infty)$ excluding a set $E$ with $\int_E dt < \infty$.
 
 In the case when $R_0<\infty,$ by  Lemma 4.1.7 in \cite{book:minru}, we get
$$T_{f,A}\left(r+{R_0-r\over eT_{f,A}(r)}\right)\leq eT_{f,A}(r)$$
for every $r \in  [0, R_0)$ excluding a set $E$ with $\int_E{dt\over R_0-t}<\infty$. We take $R = r+{R_0-r\over eT_{f,A}(r)}$. Then
\begin{equation*}
S(r,f,A) \leq K \Big(\log{1 \over R_0-r} + \log^+T_{f, A}(r)\Big)
	\end{equation*}
for every $r \in [0,R_0)$ excluding a set $E$ with $\int_E {dt \over R_0-t} < \infty$.
 \end{proof}

We obtain the following defect relation (in terms of  Nevanlinna's defect)  as a consequence of  Theorem  \ref{smtXie}.
\begin{corollary} Under the same assumptions in Theorem \ref{smtXie}.
If we further assume  either $R_0=\infty$ or $R_0<\infty$ with 
 $$\limsup_{r\rightarrow R_0} {T_{f, A}(r)\over {\log{1\over R_0-r}}}=\infty,$$
then
\begin{equation}\label{Xiedefect} \delta_{1}^{f, A, *}(D) \leq \gamma_{A, -D}- {\tilde{m}\over m},\end{equation}
\end{corollary}

\noindent{\it Proof of Theorem \ref{smtyamanoi}}
	\begin{proof}
	We take a local chart of an open affine subset $(U, (z_1, \cdots, z_n))$ such that $D|_U= \{z_1  \cdots z_s =0\}$ for some $0 \leq s \leq n$. For $s=0$, we mean $D \cap U = \emptyset$. We may assume $L\big|_U$ and $[D]\big|_U$ are trivial. Then $J_kD_1+\cdots +J_kD_q\big|_{\pi_k^{-1}(U)}$ is defined by the ideal 
$$
(z_1, dz_1,\dots, d^kz_1) \cdots(z_s, dz_s,\dots, d^kz_s).
$$
Since $J_kD_1+\cdots+J_kD_q \subset \pi_k^*(w=0)$, we have 
		\begin{equation*}
w|_U \in (z_1, dz_1,\dots, d^kz_1)\cdots(z_s, dz_s,\dots, d^kz_s).
	\end{equation*}
	Hence, by letting $I = \{0,\dots, k\}$, we have
		\begin{equation*}
	w|_U = \sum_{    \alpha \in I^s, 
 	\alpha = (\alpha _1,...,\alpha _s)} w_{\alpha}d^{\alpha_1}z_1\cdots d^{\alpha_s}z_s
	\end{equation*}
	for some $w_{\alpha}$ in $H^0(U, \mathcal O _ {J_kD})$.
	Let $s_D$ be the canonical section associated to $D$. Since $[D]\big|_U$ is trivial and $D|_U= \{z_1  \cdots z_s =0\}$,  $s_D\big|_U=z_1  \cdots z_s \cdot a$ for some nowhere zero local holomorphic function $a$ on $U$.
	Hence 
$$
{w \over s_D}\Bigg|_U = \sum_{    \alpha \in I^s, 
 	\alpha = (\alpha _1,...,\alpha _s)}
 \tilde{w}_{\alpha}\Big({d^{\alpha_1}z_1\over z_1}\Big)\cdots \Big({d^{\alpha_s}z_s\over z_s}\Big),
$$ 
where $\tilde{w}_{\alpha} = w_\alpha / a.$
	Thus  ${w \over s_D} \in H^0( X,E_{k,m}^{GG}\Omega _X( \log D ) \otimes (L \otimes [D] )^{-1} )$. Note that $w \over s_D$ has at most logarithmic pole singularities along $D$ of order at most 1.  
	
	Denote by $\| \cdot \|_{L^{-1}}$ and $\| \cdot \|_D$ the  Hermitian metrics on $L^{-1}$ and $[D]$ respectively. Then ${\| \cdot \|_{L^{-1}} \over \| \cdot \|_D~~~{~~}~~}$ is an Hermitian metric on $(L \otimes [D] )^{-1} $.  	By Lemma \ref{jetdefest} (ii), we get, with $0<mt<p<1$, 
	\begin{equation}\label{M1}
\int_0^{2\pi }  {\big\| w(j_k(f))(re^{i\theta}) \big\|_{L^{-1}}^{t} \over \big\| s_D(f)(re^{i\theta}) \big\|_D^{t}} {d\theta \over 2\pi}  \leq K   \Big( {R\over r(R-r)} T_{f, A}(R)\Big)^p.
\end{equation}
On the other hand, 
	  by applying the Poincar\'e-Lelong formula  (see Theorem A2.2.2 in \cite{book:minru}),  we have
		\begin{eqnarray*}
		dd^c \log {\big\| w(j_k(f))   \big\|_{L^{-1}}^2 \over \big\| s_D(f) \big\|_D^2 }&=& dd^c \log \big\| w(j_k(f))\big\|_{L^{-1}}^2 -dd^c \log \big\| s_D(f) \big\|_D^2 \\ 
		&=& \nu^{j_k(f)}\big((w=0)\big) - f^*c_1(L^{-1}) - \nu ^f(D) + f^*c_1([D]) \\
		&=& \nu^{j_k(f)}\big((w=0)\big) - \nu ^f(D) + f^*c_1(L \otimes [D]), 
	\end{eqnarray*}
	where $j_k(f): B(R_0)\rightarrow J_kX$ is the $k$-th lifting of $f$.
Hence, by using the Green-Jensen formula, we get
			\begin{equation}\label{Q1}
N_{j_k(f)}\big(r,(w=0)\big) - N_f(r,D) + T_{f,L \otimes [D]}(r) \leq 	\int_0^{2\pi } \log {\big\| w(j_k(f))(re^{i\theta}) \big\|_{L^{-1}} \over \big\| s_D(f)(re^{i\theta}) \big\|_D } {d\theta \over 2\pi} + O(1).
	\end{equation}
	Following the same argument as in the proof of Theorem \ref{smtXie} by using  (\ref{M1}), we can get
	\begin{equation}\label{Q2}
	T_{f, L \otimes [D]}(r) \leq N_f(r, D) - N_{j_k(f)}(r, (w=0)) + S(r,f,A),
\end{equation}where $S(r, f, A)$ is evaluated in Theorem \ref{smtXie}.

By the assumption that $J_kD_1+\cdots+J_kD_q\subset \pi_k^*(w=0)$, we have 
	\begin{equation}\label{Q3} N_f(r,D) - N_{j_k(f)}\big(r,(w=0)\big) 
		\leq \sum\limits_{i = 1}^q \Big[ N_f(r,D_i) - N_{j_k(f)}(r, J_kD_i)\Big].\end{equation}
		Here we regard $J_k D_i$ as a subscheme of $J_kX$, and the counting function for the subscheme 
		is defined according to Yamanoi \cite{Yamanoi2}. Namely, 
		$$N_{j_k(f)}(r, J_kD_i) =\int_{r_0}^r \left (\sum_{z\in B(t)} \ord_z  (j_k^*(f) J_k(D_i))\right){dt\over t},$$
for some fixed $r_0$ with $0<r_0<r<R_0.$
				
		We claim that, for $1\leq i\leq q$ and  $z\in B(R_0)$ such that $f(z)\in \text{supp} D$, 
	\begin{equation}\label{multi}
	\ord_z (j^*_k(f)J_k(D_i))= \ord_z f^*D_i-\min\{k, \ord_z f^*D_i\}.
	\end{equation}
	Indeed, let  $V\subset X$ be a Zariski open subset such that
$f(z)\in V$.    Write $D_i|_V=\{g_i=0\}$, where $g_i\in H^0(V, {\mathcal O}_X)$ is the local defining function of $D_i$. Then
$$\ord_z f^*D_i=\ord_z(g_i\circ f)$$
and $$\ord_z (j^*_k(f) J_k(D_i))=\min\{\ord_z g_i\circ f, \dots, \ord_z  (j^*_k(f)(d^k g_i))\},$$
where, by the definition,  $$
j^*_k(f)(d^l g_i)={d^l\over dz^l}(g_i\circ f)$$
for $l=1, \dots, k$. Hence
\begin{eqnarray*}\min\{\ord_z g_i\circ f, \dots, \ord_z  (j^*_k(f)(d^k g_i))\}&=&\max\{\ord_z (g_i\circ f)-k, 0\}\\
&=& \ord_z g_i\circ f -\min\{k, \ord_z g_i\circ f\}.
\end{eqnarray*}
Therefore
$$\ord_z (j^*_k(f) J_k(D_i)) = \ord_z g_i\circ f -\min\{k, \ord_z g_i\circ f\} = \ord_z f^*D_i-\min\{k, \ord_z f^*D_i\}.$$
This proves our claim.  The claim implies that 
\begin{equation}\label{Q4}
N_f(r,D_i) - N_{j_k(f)}(r, J_kD_i)\leq N^{(k)}_f(r,D_i) .\end{equation}
Combing (\ref{Q1}), (\ref{Q2}), (\ref{Q3}) and (\ref{Q4}) gives 
		\begin{equation*}
	T_{f,L \otimes [D]}(r) \leq \sum\limits_{i = 1}^q N^{(k)}_f(r,D_i) + S(r,f,A).
	\end{equation*}
\end{proof}
Theorem \ref{smtyamanoi} gives the following defect relation, in terms of the Nevanlinna's defect.
\begin{corollary}\label{Yamanoidef} Under the same assumptions in Theorem  \ref{smtyamanoi}.
If we further assume  either $R_0=\infty$ or $R_0<\infty$ with 
 $$\limsup_{r\rightarrow R_0} {T_{f, A}(r)\over {\log{1\over R_0-r}}}=\infty,$$
then
\begin{equation}\label{Ydefect}\sum_{j=1}^q \delta_{k}^{f, A, *}(D_j) \leq \gamma_{A, L}. \end{equation}
\end{corollary}

\noindent{\it Proof of Theorem \ref{yamanoinonint}} \begin{proof} 
	Since the universal covering of $M$ is $B(R_0) \subset \mathbb C$, by lifting $f$ to $B(R_0),$ we may assume that $M = B(R_0)$. In the case for $R_0=\infty$ and the case for $R_0<\infty$ but
	\begin{equation*}
\limsup_{r\rightarrow R_{0}}  {T_{f,A}(r) \over  \log {1 \over R_0 -r}} = \infty, 
	\end{equation*}
	since from Proposition \ref{defectcomparision}, 
	$\delta_{\mu_0}^{f, A }(D_j) \leq 	\delta_{\mu_0}^{f, A, * }(D_j)$, we get, from (\ref{Ydefect}),
		$$\sum_{j=1}^q   \delta_{k}^{f, A}(D_j)
	 \leq  \gamma_{A, L},$$
			which proves  this case.
So we may assume 
	\begin{equation}\label{ass1}
		\limsup_{r\rightarrow R_{0}}  {T_{f,A}(r) \over  \log {1 \over R_0 -r}} < \infty. 
	\end{equation}
	We also assume $R_0 =1$. 
		We prove by contradiction for this case. So we suppose that 
	$$\sum_{j=1}^q   \delta_{k}^{f, A}(D_j)
	 >  \gamma_{A, L} +2\rho m.$$
			By the definition of non-integrated defect, for each $1 \leq j \leq q$, there exists a real number $\delta_j \geq 0$ and a bounded continuous real-valued function $h_j$ on $B(1)$ such that 
		\begin{equation}\label{r1}
		f^*c_1([D_j])- \delta_j f^*c_1([A]) +dd^c\log h_j^2 \ge  \left[\min \{\nu^f(D_j), k\}\right],
	\end{equation}
	and 
	\begin{equation}\label{key}
		\sum\limits_{j = 1}^q  \delta_j > \gamma_{A, L} + 2\rho m.
	\end{equation}
We define, for an arbitrary effective  divisor $D$ on $X$, 
 \begin{equation}\label{a1}
 \big\|f\big\|_{D} := \frac{\big|\psi\big|}{\big\|s_{D}(f)\big\|},
 \end{equation}  
 where $s_D$ is the canonical section of $[D]$ and $\psi$ is a holomorphic function on $B(1)$ such that $\nu_\psi = \nu^f(D).$
Then, from (\ref{Chern}), we have  $
f^*c_1([D_j]) = dd^c \log  \big\|f\big\|_{D_j}^2$, and $f^*c_1([A]) = dd^c \log  \big\|f\big\|_{A}^2$. 
	Hence we can rewrite (\ref{r1}) as
\begin{equation}\label{Dj}
dd^c\log \big\|f\big\|_{D_j}^2 - \delta_j dd^c\log\big\|f\big\|_A^2 + dd^c\log h_j^2 \geq dd^c\log\big|\varphi_j\big|^2,
\end{equation}
where  $\varphi_j$ is a holomorphic function on $B(1)$ such that  $\nu_{\varphi_j} = \min\{\nu^f(D_j), k\}.$ 
Thus,  for $1\leq j \leq q,$ there exists a subharmonic function $u_j\not\equiv 0$ on $B(1)$ such that
\begin{equation}\label{uj}
e^{u_j} \leq {\big\|f\big\|_{D_j}\over \big\|f\big\|^{\delta_j}_A}
\end{equation}
and $u_j - \log \varphi_j$ is subharmonic.
Note that if we let
$$
\big\|f\big\|_D = \big\|f\big\|_{D_1}\cdots \big\|f\big\|_{D_q},
$$	
then $f^*c_1([D]) = dd^c\log \big\|f\big\|_D^2$ and
\begin{equation}\label{u}
  e^{\sum_{j=1}^{q}u_j} \leq {\big\|f\big\|_D\over \big\|f\big\|^{\sum_{j=1}^{q}\delta_j}_A}.
\end{equation}	
	From the proof of Theorem \ref{smtyamanoi}, we know that
	$${\mathcal P \over s_D} \in H^0\left( X,E_{k,m}^{GG}\Omega _X( \log D )\otimes(L\otimes [D])^{-1} \right).$$	
Hence Lemma \ref{jetdefest} (ii) implies that, for $0<tm<p<1$, 
$$
\int_0^{2\pi } \left\|{\mathcal P(j_k(f))(re^{i \theta })\over s_D(f(re^{i\theta}))}\right\|_{h^{-1}}^t{d\theta \over 2\pi}\leq K \Big( {R\over r(R-r)} T_{f, A}(R)\Big)^p,$$
for $0<r_0<r<R<1$, 
where $h$ is an Hermitian metric on $L\otimes[D].$
According to Lemma 4.1.7 in \cite{book:minru2}, choose $R := r + {1-r\over eT_{f,A}(r)},$
then $
T_{f,A}(R) \leq 2T_{f,A}(r)$
 outside a set $E$ with $\int_E{1\over 1-r}dr<\infty.$ Thus, by also using (\ref{ass1}), we get, for $0<tm<p<1$, 
\begin{eqnarray*}
\int_0^{2\pi } \log \left\|{\mathcal P(j_k(f))(re^{i \theta })\over s_D(f(re^{i\theta}))}\right\|_{h^{-1}}^t{d\theta \over 2\pi}
&\leq& K_1 {1\over (1-r)^p}\big(T_{f,A}(r)\big)^{p}\\
&\leq& K_2 {1\over (1-r)^p}\left(\log {1\over 1-r}\right)^{p},
\end{eqnarray*}
for all $0<r_0\leq r<1$ outside a set $E$ with $\int_E{1\over 1-r}dr<\infty.$ 
By Proposition 4.1.8 in \cite{book:minru2}, varying a constant $K_2$ slightly, we may assume that the above inequality holds for all $r\in [0,1).$
Therefore
\begin{equation}\label{keyest}
		\int_0^{2\pi } \left\|{\mathcal P(j_k(f))(re^{i \theta })\over s_D(f(re^{i \theta }))}\right\|_{h^{-1}}^t{d\theta \over 2\pi} \leq K \left( {1 \over 1-r} \right) ^{p}  \left( \log {1 \over 1-r} \right)^{p}
	\end{equation}
with a suitable positive constant $K$ for all $r\in [0,1).$

Let 
\begin{equation}\label{v}
		v :=  \log \left( \left\| { \mathcal P(j_k(f))  \over s_D(f)}\right\|_{h^{-1}}  \cdot {\big\|f\big\|_A^{\gamma_{A,L}}\over \big\|f\big\|_D}\right)+  \sum_{j=1}^q u_j.
\end{equation}
	We first show that $v$ is subharmonic. Indeed, from the proof of Theorem \ref{smtyamanoi} (especially by using (\ref{multi})), we have 
\begin{eqnarray*}
2 dd^c \log   \left\| { \mathcal P(j_k(f))  \over s_D(f)}\right\|_{h^{-1}} 
&\ge& f^*c_1(L \otimes [D]) - \sum_{j=1}^q \min\{\nu^f(D_j), k\}\\
&=&f^*c_1(L) + f^*c_1([D]) - \sum_{j=1}^q \min\{\nu^f(D_j), k\}.
\end{eqnarray*}
Hence
\begin{eqnarray*}
	2 dd^c v
&\ge& f^*c_1(L) + f^*c_1([D]) - \sum_{j=1}^q \min\{\nu^f(D_j), k\} \\
&~& + ~ \gamma_{A,L}f^*c_1([A]) - f^*c_1([D]) + \sum_{j=1}^{q}2dd^cu_j \\
&=& f^*\big(\gamma_{A,L}c_1([A]) + c_1(L)\big) + \sum_{j=1}^{q}dd^c(2u_j-\log |\varphi_j|^2) \\
&~&+ ~ \sum_{j=1}^{q} dd^c\log |\varphi_j|^2 - \sum_{j=1}^q \min\{\nu^f(D_j), k\}\ge 0.
\end{eqnarray*}
So $v$ is subharmonic. 
By the growth condition of $f$, there exists a subharmonic function $w \not \equiv 0$ on $B(1)$ such that 
\begin{equation}\label{w}
\lambda e^w  \leq  \big\|f\big\|_A^{\rho},
\end{equation}
	where $\omega ={\sqrt{-1}\over 2\pi} \lambda dz\wedge d\bar{z}$ is the metric form on $M$.
	Let 
	\begin{equation*}
	 u=w+tv
	\end{equation*}
	where $t$ is a positive number to be determined later.  By (\ref{v}) and (\ref{w}),we get
	\begin{eqnarray*}\label{int}
		 e^u \lambda^2
&\leq&    \left\| { \mathcal P(j_k(f))  \over s_D(f)}\right\|^t\cdot {\big\|f\big\|_A^{t\gamma_{A, L}}\over \big\|f\big\|_D^t}\cdot e^{\sum_{j=1}^q tu_j } \cdot \big\|f\big\|_A^{2\rho} \\
&\leq&   \left\| { \mathcal P(j_k(f))  \over s_D(f)}\right\|^t\cdot {\big\|f\big\|_A^{t\gamma_{A, L}}\over \big\|f\big\|_D^t}\cdot {\big\|f\big\|_D^t\over \|f\|_A^{t\sum_{j=1}^q\delta_j}} \cdot \big\|f\big\|_A^{2\rho} \\
&=&  \left\| { \mathcal P(j_k(f))  \over s_D(f)}\right\|^t\cdot  \big\|f\big\|_A^{t(\gamma_{A, L}-\sum_{j=1}^q\delta_j)+2\rho}.
	\end{eqnarray*}
	By letting
$$
t\Big(\gamma_{A, L}-\sum_{j=1}^q\delta_j\Big)+2\rho = 0,
$$
we get, from above, 
$$ e^u \lambda^2\leq  \left\| { \mathcal P(j_k(f))  \over s_D(f)}\right\|^t.$$
	We check that $tm < 1.$ Indeed, by (\ref{key}), 
$$
tm = {2\rho m\over\sum_{j=1}^q\delta_j-\gamma_{A, L}} < {2\rho m\over \gamma_{A, L} + 2\rho m - \gamma_{A, L}} = 1.
$$	
	Hence, by (\ref{keyest}), 
	\begin{equation*}
		\int_0^{2\pi}   (e^u\lambda^2)(re^{i\theta}) d\theta  \leq K {1 \over (1-r)^{p}} \left(\log {1 \over 1-r} \right)^{p} 
		\leq K {1 \over (1-r)^{p'}},
	\end{equation*}
where $p<p'<1$.
Therefore
\begin{equation*}
		\int_{B(1)}   (e^u\lambda^2)(re^{i\theta}) rdr  d\theta  \leq K \int_0^1 {rdr\over (1-r)^{p'}} < \infty.	\end{equation*}
On the other hand,   by the result of \cite{Yau76} (or \cite{Karp82}), we have necessarily
	\begin{equation*}
		\int_{B(1)}  e^u  dV =  \infty ,
	\end{equation*}
	which is a contradiction. This completes the proof.
\end{proof}

We recall the following Proposition  from \cite{Brotbek} (For notations, see \cite{Brotbek}).
\begin{proposition}\label{Wronskian}(\cite{Brotbek}, Proposition 2.3). 
For any $s_0, \dots,,s_k \in H^0(X, L)$, 
the locally defined jet differential equations $W_U(s_0,...,s_k)$ 
 glue together into a section 
 $$W(s_0, \dots, s_k)\in H^0(X, E_{k,{k(k+1) \over 2}}^{GG}\Omega_X\otimes L^{k+1}).$$
 \end{proposition}

In our case,  we let $X={\Bbb P}^n(\Bbb C)$, $D=H_1+\cdots +H_q$, where $H_1, \dots, H_q$ are  the hyperplanes in ${\Bbb P}^n(\Bbb C)$ in general position, and
$L:=K_{{\Bbb P}^n(\Bbb C)}={\mathcal O}_{{\Bbb P}^n({\Bbb C})}(-(n+1))$, the canonical line bundle of ${\Bbb P}^n(\Bbb C)$. Let 	\begin{equation*}
\mathcal P: =W(z_0, \dots, z_n)= \det \left( {\begin{array}{*{20}{c}}
	z_0 &  \ldots  & z_n \\	
	dz_0 &  \ldots  & dz_n \\
	\vdots  &  \ddots  &  \vdots   \\
	{d^n z_0} &  \cdots  & {d^ nz_n}  \\
	
	\end{array} } \right),
\end{equation*}
where $z_0, \cdots, z_n$ are global sections of $\mathcal O_{\mathbb P ^n(\mathbb C)} (1) $. Then Proposition \ref{Wronskian} implies that 
  $\mathcal P\in H^0 \Big( X,E_{n,{n(n+1) \over 2}}^{GG}\Omega _{\mathbb P ^n (\mathbb C)} \otimes L^{-1}\Big)$.
 
 We  now show that $J_nH_1 + \cdots + J_nH_q \subset \pi_n^*(\mathcal P = 0)$.
Fix arbitrary point $x$ in $\mathbb P ^n (\mathbb C)$. There exists a linear change of coordinates $w_0, \cdots, w_n$ of  $z_0, \cdots, z_n$, i.e., $w_i = \sum\limits_{j = 0}^n a_{ij}z_j$ for some $a_{ij} \in \mathbb C$ with $\det (A) \ne 0 $ where $A = (a_{ij})$, such that in a local chart $V$ around $x$, we have $D|_V = \{w_0 \cdots w_t = 0\}$ for some $0 \leq t \leq n$. Note that we have
	\begin{equation}\label{Wronchange}
\mathcal P =\det(A)^{-1} \cdot \det \left( {\begin{array}{*{20}{c}}
w_0 &  \ldots  & w_n \\	
	dw_0 &  \ldots  & dw_n \\
	\vdots  &  \ddots  &  \vdots   \\
	d^n w_0 &  \cdots  & d^n w_n \\
	
	\end{array} } \right).
\end{equation}
Also note that $J_nH_1 + \cdots + J_nH_q \big|_{\pi_n^{-1}(V)}$ is defined by the ideal 
\begin{equation*}
	(w_1, dw_1, \cdots, d^n w_1) \cdot 	(w_2, dw_2, \cdots, d^n w_2) \cdots (w_t, dw_t, \cdots, d^n w_t).
\end{equation*}
Then it is easy to see that, from (\ref{Wronchange}),
\begin{equation*}
\mathcal P |_V \in (w_1, dw_1, \cdots, d^n w_1) \cdot 	(w_2, dw_2, \cdots, d^n w_2) \cdots (w_t, dw_t, \cdots, d^n w_t).
\end{equation*}
Hence we have $J_nH_1 + \cdots + J_nH_q \subset \pi_n^*(\mathcal P = 0)$. 
Thus Theorem \ref{yamanoinonint}  with this jet differential $\mathcal P$ gives
 Fujimoto's result (Theorem \ref{fujimoto thm}).

\section{Non-integrated defect relation for holomorphic maps into ${\Bbb P}^n(\Bbb C)$ intersecting high degree generic hypersurfaces}
We briefly recall the construction of Demaily-Semple jet tower as follows, we refer the readers \cite{rosseauimpa} or the original paper of Demailly \cite{Dem} for details. Let ${\mathbb G}_k$ be the group of germs of $k$-jets of biholomorphisms of
$({\mathbb C}, 0)$, that is, the group of germs of biholomorphic maps
$\phi: t \mapsto a_1 t + a_2 t^2+\cdots+a_kt^k$, $a_1\in {\mathbb C}^*$, $a_j\in {\mathbb C}$ for $j\geq1$,
in which the composition law is taken modulo terms $t^j$
of degree $j>k$. The group $\mathbb{G}_k$ has a natural action on $J_k(X)$ given by $\phi\cdot j_k(f):=j_k(f\circ \phi)$ for any $\phi \in \mathbb{G}_k$.
Denote by
$$J^{\text{reg}}_k X:=\{j_k(f)\in J_k X~|~f'(0)\not=0\},$$
the space of regular jets.
 Consider the pairs $(X', V')$, where $X'$ is a complex manifold, $V' \subset T_{X'}$ is a subbundle, and denote by $\pi':V'\to X'$ the natural projection. Starting from $(X,T_X)$, define $X_1 := {\mathbb P}(T_X)$, let $\pi_{0,1}:X_1\to X$ the natural projection.  Define the bundle $V_1 \subset T_{X_1}$  fiberwise by
 $$
 V_{1,(x,[v])} := \{w \in   T_{X_1, (x,[v])} | (d\pi_{0,1})_x(w)\in {\mathbb C}v\}.$$
In other words,  $V_1$ is characterized by the exact sequence
 $$0\rightarrow T_{X_1/X}\rightarrow V_1\xrightarrow{d\pi_{0,1}} {\mathcal O}_{X_1}(-1)\rightarrow 0$$
 where ${\mathcal O}_{X_1}(-1)$  is the tautological line bundle on $X_1$, and $T_{X_1/X}$  is the
 relative tangent bundle corresponding to the fibration $\pi_{0, 1}$.
Inductively by this procedure,  
we get the {\it Demailly-Semple tower}
$$
\xymatrixcolsep{2pc}\xymatrix{
	  (X_k,V_k)\ar[r]^-{\pi_{k-1,k}} &(X_{k-1},V_{k-1})\ar[r]^-{\pi_{k-2,k-1}}&\cdots \rightarrow (X_1,V_1)\ar[r]^{\pi_{0,1}}&(X,T_X).
	  }
$$
Then we have an embedding $J^{reg}_k X/{\mathbb G}_k \rightarrow X_k$.
Denote by $X_k^{\text{reg}}$ the image of this embedding in $X_k$ 
and denote by $X_k^{\text{sing}}:=X_k\backslash X_k^{\text{reg}}$. Then $X_k^{\text{sing}}$ is a divisor in $X_k$. 
Denote by  $\pi_k: X_k \rightarrow  X$ the projection, then the direct
image sheaf $\pi_{k*}{\mathcal O}_{X_k}(m)$ is isomorphic to ${\mathcal E}_{k,m}\Omega_X,$ i.e., 
$$
\pi_{k*}{\mathcal O}_{X_k}(m)\simeq{\mathcal E}_{k,m}\Omega_X,
$$  
whose
sections are precisely the invariant jet differentials $P$ of order $k$ and degree $m$, i.e., for any $g\in \mathbb{G}_k$ and any $j_k (f)\in J^{reg}_k (X)$, 
$P(j_k(f\circ g))=g'(0)^m P(j_k (f)).$ We shall denote the associated
vector bundle by $E_{k, m}\Omega_X$.  The fiber of $\pi_k$ at a non-singular point of $X$ is denoted by ${\mathcal R}_{n,k}$, 
it is a rational manifold, and it is a compactification of the quotient
$({\mathbb C}^{nk}-\{0\})/{\mathbb G}_k$.

 The above definition can be also extended to the logarithmic setting. 
We now briefly recall the construction of the logarithmic version of Demaily-Semple tower due to Dethloff-Lu \cite{DL01}. A
 {\it logarithmic directed manifold} is a triple $(X, D, V)$
where $(X,D)$ is a log-manifold, $V$ is a subbundle of $T_X(-\log D)$. 
A morphism between logarithmic directed manifolds $(X',D',V')$ and $(X,D,V)$ is given by a holomorphic map $f:X'\to X$ such that $f^{-1}D\subset D'$ and $f_*V'\subset V$.  
 The  logarithmic Demailly-Semple $k$-jet tower
 \begin{eqnarray*}
 (X_k(D), D_k, V_k) &\xrightarrow{\pi_{k-1,k}}& (X_{k-1}(D), D_{k-1}, V_{k-1}) \xrightarrow{\pi_{k-2,k-1}}
 \cdots \\
 &\rightarrow&  (X_1(D), D_1, V_1) \xrightarrow{\pi_{0,1}} (X,D, T_X(-\log D))
 \end{eqnarray*}
 is constructed  inductively as follows:  Starting from
 $V_0=V=T_X(-\log D)$, define
$X_k(D):={\mathbb P}(V_{k-1})$, and let $\pi_{k-1, k}: X_k(D)\rightarrow X_{k-1}(D)$
be the natural projection. Set $D_k:=(\pi_{k-1, k})^{-1}(D_{k-1})$
which is a simple
normal crossing divisor. Note that $\pi_{k-1, k}$  induces a morphism
$$(\pi_{k-1, k})_*: T_{X_k(D)}(-\log D_k)\rightarrow (\pi_{k-1, k})^*T_{X_{k-1}(D)}(-\log D_{k-1}).$$
Define $V_k: =(\pi_{k-1, k})_*^{-1}
{\mathcal O}_{X_k(D)}(-1)\subset T_{X_k(D)}(-\log D_k)$, where 
${\mathcal O}_{X_k(D)}(-1): = {\mathcal O}_{{\mathbb P}(V_{k-1})}(-1)$
is the tautological line bundle, which by definition is also a subbundle of
$\pi_{k-1,k}^*V_{k-1}$.  In our case, denoting by $X_k(D)$ the log Demailly $k$-jet tower associated to $(X, D,T_X(-\log D)).$
By Proposition 3.9 in \cite{DL01}, we have
$${\mathcal E}_{k,m}\Omega_X(\log D) =(\pi_k)_*{\mathcal O}_{X_k(D)}(m),$$ 
where  ${\mathcal E}_{k,m}\Omega_X(\log D)$  is the subsheaf of 
${\mathcal E}^{GG}_{k,m}\Omega_X(\log D)$ consisting of
invariant logarithmic differential operators $P$ of order $k$ and degree $m$.

Let $X$ be a smooth projective variety and $D$ be a Cartier divisor on $X$. Recall that the 
stable base locus of $D$ is 
defined by $B(D):=\bigcap_{m\in {\mathbb N}} Bs(|mD|)$. We recall the following result which is due to D. Brotbek and Y. Deng. 
\begin{theorem}[\cite{BD2}, Corollary 4.5]\label{cor4.5}
	Let $X$ be a smooth projective variety of dimension $n$. Let $A$ be a very ample line bundle on $X$. Let $k,$ $k',$ $\epsilon,$ $\delta,$ $r$ be positive integers such that $k=n+1,$ $k'={k(k+1) \over 2},$ $\epsilon \geq k,$ $\delta=(k+1)n+k,$ and $r> \delta^{k-1}k(\epsilon+k\delta)$. Then there exist $\beta, \tilde \beta \in \Bbb N$ such that for any $\alpha \geq 0$, and for a generic hypersurface $D \in \big|A^{\epsilon+(r+k)\delta}\big|$, we have
	\begin{equation*}
	 B \left(\mathcal O_{X_k(D)}(\beta+ \alpha \delta^{k-1}k') \otimes \pi _{0,k}^*A^{\tilde \beta+ \alpha(\delta^{k-1} k(\epsilon+ k\delta)-r)}\right)
	\subset X_k(D)^{sing} \cup \pi _{0,k}^{-1}(D).
	\end{equation*}
	\end{theorem}

We apply the above theorem to obtain the following result.
\begin{proposition}\label{DB}
	Let $X$ be a smooth projective variety of dimension $n \geq 2$. Let $A$ be a very ample line bundle on $X$. Fix a positive integer $c$. Let $D \in \big|A^d\big|$ be a generic smooth hypersurface in $X$ with 
	$$d \geq (n+1)^{n+3}\Big(n+1+{c \over 2}\Big)^{n+3}.$$
	Let $f: B(R_0)\subset {\Bbb C}\rightarrow X$, $0<R_0 \leq \infty$, be a non-constant holomorphic map with $f(B(R_0)) \not\subset  \text {supp}(D)$. Then,  for jet order $k=n+1$, there exists some global logarithmic jet differential
		$$\mathcal P \in H^0\left( X,E_{k,m}^{GG}\Omega _X( \log D ) \otimes A^{- \tilde m} \right)$$
	such that
	$$\mathcal P (j_k(f)) \not  \equiv 0,$$
	for some $m,\tilde m \in \mathbb N$ with $\tilde m > cm$.
	\end{proposition}

\begin{proof} 
	For the convenience of computation, we may assume $c \geq 3$. For the cases $c=1$ and $c=2$, please see Corollary 4.9 in \cite{BD2} and Theorem 2.2 in \cite{H} respectively.
	
	Let $k=n+1, k'={k(k+1) \over 2}, \delta=(k+1)n+k=n^2+3n+1$ and let $r_0=c\delta ^{k-1}k'+\delta^{k-1}(\delta +1)^2=\delta^{k-1}(\delta +1)\Big(\delta +1+ {c \over 2}\Big)$. By computation, we have $k(k+\delta -1 +k\delta)<(\delta +1)^2$. Hence, every integer $d \geq (r_0+k)\delta+2\delta $ can be written as
$$d=\epsilon+(r+k)\delta $$
where $k \leq \epsilon \leq k+\delta -1$ and $r > c \delta ^{k-1}k'+ \delta ^{k-1}k(\epsilon + k  \delta) $.

Applying Theorem \ref{cor4.5} for $\alpha$ large enough such that $\alpha(r-\delta^{k-1}k(\epsilon+k\delta))-\tilde \beta >0$, we see that there exists a global logarithmic jet differentical
$$\mathcal P_\alpha \in H^0\left( X,E_{k,m(\alpha)}^{GG}\Omega _X( \log D ) \otimes A^{- \tilde m(\alpha)} \right)$$
such that $\mathcal P _\alpha (j_k(f)) \not  \equiv 0$, where $$m(\alpha) = \beta + \alpha\delta^{k-1}k'$$ and $$\tilde m(\alpha) = \alpha(r-\delta^{k-1}k(\epsilon+k\delta))-\tilde \beta. $$ 

Note that 
\begin{equation*}
\mathop {\lim }\limits_{\alpha  \to \infty } {\tilde m(\alpha) \over m(\alpha)} = \mathop {\lim }\limits_{\alpha  \to \infty }{\alpha(r-\delta^{k-1}k(\epsilon+k\delta))-\tilde \beta  \over \beta+\alpha\delta^{k-1}k'} = {r-\delta^{k-1}k(\epsilon+k\delta) \over \delta^{k-1}k'} > c.
\end{equation*}
	So we can take $\tilde m = \tilde m(\alpha),$ $m = m(\alpha)$ and $\mathcal P =\mathcal P_\alpha$ for large enough $\alpha \gg 0$. 
	
	To complete the proof, it suffices to give an upper bound of $(r_0+k)\delta + 2\delta$. Recall that $k=n+1, k'={k(k+1) \over 2}, \delta=(k+1)n+k=n^2+3n+1$ and $r_0=\delta^{k-1}(\delta +1)\Big(\delta +1+ {c \over 2}\Big)$. Then we have
		\begin{eqnarray*}
		(r_0+k)\delta + 2\delta &=& (r_0+k+2)\delta  \\ 
			&=& 	\Big[\delta^{k-1}(\delta+1)\Big(\delta +1+ {c \over 2}\Big) +k+2\Big]\delta\\
			&=& 	\Big[(n^2+3n+1)^n(n^2+3n+2)\Big(n^2+3n+2+{c \over 2}\Big)+n+3\Big]\cdot\\
        &~&(n^2+3n+1)\\
		&\leq& (n^2+3n+2)^n(n^2+3n+2)\Big(n^2+3n+2+{c \over 2}\Big)(n^2+3n+2)\\
		&=& (n^2+3n+2)^{n+2}\Big(n^2+3n+2+{c \over 2}\Big).
	\end{eqnarray*}
	Note that $(n+1)\Big(n+1+{c \over 2}\Big) = n^2+  \Big(2+{c \over 2}\Big)n +\Big(1+{c \over 2}\Big) $
	
	Since $n \geq 2$ and $c \geq 3$, we have
	\begin{equation*}
n^2+  \Big(2+{c \over 2}\Big)n +\Big(1+{c \over 2}\Big) > n^2+3n+2+{c \over 2}.
	\end{equation*}
	Thus, we have
	\begin{equation*}
(r_0+k)\delta + 2\delta < (n+1)^{n+3}\Big(n+1+{c \over 2}\Big)^{n+3} .
	\end{equation*}
\end{proof}

\noindent{\it Proof of Theorem \ref{Brotbekdefect}}.

\begin{proof} (i) 
	By Proposition \ref{DB}, for jet order $k=n+1$, there exists some global jet differential 
	\begin{equation*}
	\mathcal P \in H^0\Big( X,E_{k,m}^{GG}\Omega _X (\log D)\otimes [{- \tilde m}A] \Big)
\end{equation*}
	such that 
	\begin{equation*}
	\mathcal P (j_k(f)) \not \equiv 0
\end{equation*}
	for some $m, \tilde m \in \mathbb N$ with $\tilde m > cm$. By Theorem \ref{smtXie}, we have 
		\begin{equation}\label{HXie}
	\tilde m T_{f, A}(r) \leq mN^{(1)}_f(r, D) + S(r,f,A).
	\end{equation}  
	In the case when $R_0 = \infty$ and in the case when $R_0<\infty$ but
		\begin{equation*}
		\limsup_{r\rightarrow R_{0}}  {T_{f, A}(r) \over  \log {1 \over R_0 -r}} = \infty,
			\end{equation*}
	we obtain from (\ref{HXie}), noticing that $D\in |dA|$,  
			\begin{equation*}
	\delta _1^{f,A, *}(D)  \leq d - {\tilde{m}\over m} < d-c.
	\end{equation*}
	This proved (i).
	
	We now prove (ii). Since the universal covering of $M$ is $B(R_0) \subset \mathbb C$, by lifting $f$ to $B(R_0),$ we may assume that $M = B(R_0)$.
 Note that 	$\delta _1^{f,A}(D) \leq \delta _1^{f,A, *}(D)$, so the conclusion is true when 
	either  $R_0 = \infty$ or $R_0<\infty$ but 
		\begin{equation*}
		\limsup_{r\rightarrow R_{0}}  {T_{f, A}(r) \over  \log {1 \over R_0 -r}} = \infty.
			\end{equation*}
	
			So we may assume 
		\begin{equation}\label{ass2}
		\limsup_{r\rightarrow R_{0}}  {T_{f,A}(r) \over  \log {1 \over R_0 -r}} < \infty.
	\end{equation}
	We assume that $R_0=1$. 
	We prove this case by contradiction.  Since $\tilde m > cm$, it suffices to derive a contradiction from
			\begin{equation*}
	\delta _1^{f,A}(D) >	d - {\tilde m \over m} + 2\rho .
\end{equation*}
	So we assume that $	\delta _1^{f,A}(D) > d - {\tilde m \over m} + 2\rho $. Then there is $\delta \geq 0$ satisfying
	$$\delta > d - {\tilde m \over m} + 2\rho$$
and 	 
\begin{equation}\label{brotnonintassumption}
f^* c_1([D]) - \delta f^* c_1([A]) + dd^c\log h^2 \geq \big[\text{min}\{\nu^f(D), 1\}\big],
\end{equation}
where $h$ is a bounded non-negative continuous function on $B(1)$. 
Let $\big\|f\big\|_A$ be defined in (\ref{a1}). Then by (\ref{Chern}), $ f^*c_1([A]) = dd^c \log  \big\|f\big\|_A^2$.  Hence (\ref{brotnonintassumption})  can be rewritten as, by noticing that $\big\|f\big\|_D=\big\|f\big\|_A^d$, 
		\begin{equation*}
		dd^c \log  \big\|f\big\|_A^{2d} - \delta  dd^c \log \big\|f\big\|_A^2 + dd^c \log h^2\geq dd^c \log \big|\varphi\big|^2 
	\end{equation*}
where $\varphi$ is a holomorphic function on $B(1)$ with $\nu_\varphi = \text{min}\{\nu^f(D), 1\}.$
	Thus, there is a subharmonic function $u \not \equiv 0$ on $B(1)$ such that 
		\begin{equation*}
		e^u \leq \big\|f\big\|_A^{d - \delta} 
	\end{equation*}
	and $u - \log \big|\varphi\big|$ is also subharmonic.
	Let 
		\begin{equation*}
		v := \log  { \big\| \mathcal P(j_k(f))\big\| _ {\tilde h ^{-1}} \over \big\|f\big\|_A^{\tilde m}   } + mu
	\end{equation*}
where $\|\cdot\|_{\tilde h}$ is an Hermitian metric on $A^{\tilde m}.$
	We now show that $v$ is subharmonic. From the proof of Theorem \ref{smtXie}, we have 
		\begin{equation*}
		dd^c \log \big\| \mathcal P(j_k(f))\big\| _ {\tilde h ^{-1}}^2  \geq \tilde m f^*c_1([A]) - m \cdot \min \{\nu ^f(D), 1\}.
	\end{equation*}
	Hence
	\begin{eqnarray*}
	2 dd^c v&=& dd^c \log \big\| \mathcal P(j_k(f))\big\| _ {\tilde h ^{-1}}^2 - dd^c \log  \big\|f\big\|_A^{2\tilde m} + mdd^c (2u  -\log \big|\varphi\big|^2) \\
&~&+~mdd^c \log \big|\varphi\big|^2\\
	&\geq& 	\tilde m f^*c_1([A]) -  m \cdot \min \{\nu ^f(D), 1\} - \tilde m f^*c_1([A])  + mdd^c (2u  -\log \big|\varphi\big|^2)\\\
	&~& +~ m \cdot \min \{\nu ^f(D), 1\}= mdd^c (2u  -\log \big|\varphi\big|^2)\geq 0.
\end{eqnarray*}
	So $v$ is subharmonic. 
	By the growth condition of $f$, there exists a subharmonic function $w \not \equiv 0$ on $B(1)$ such that 
	\begin{equation*}
		e^w \lambda \leq \big\|f\big\|_A^{\rho} 
	\end{equation*}
	where $\omega= {\sqrt{-1}\over2\pi} \lambda dz\wedge d\bar{z}.$
	Let
	\begin{equation*}
		\tilde u=w+tv
	\end{equation*}
	where $t$ is a positive number to be determined later.  Then
	\begin{equation*}
		e^{\tilde u} \lambda^2 \leq  { \big\|\mathcal P(j_k(f))\big\| _ {\tilde h ^{-1}} ^t \over \big\|f\big\|_A^{t\tilde m}  }\big\|f\big\|_A ^{tm(d-\delta )+2 \rho }.
	\end{equation*}
	By taking 
	\begin{equation*}
	tm(d-\delta )+2 \rho =t\tilde m,
	\end{equation*}
		we get
	\begin{equation}\label{etildeu}
		e^{\tilde u}\lambda^2 \leq  \big\|\mathcal P(j_k(f))\big\| _ {\tilde h ^{-1}} ^t.
	\end{equation}
	Next we  check that $t m< 1$. By our assumption $\delta > d - {\tilde m \over m} +2\rho$,
	\begin{equation*}
		tm = {2\rho m\over  \tilde m -md +m\delta} < {2\rho m\over  \tilde m -md +m(d - {\tilde m \over m} +2\rho)}  =1.
	\end{equation*} 
Thus Lemma \ref{jetdefest} (ii) implies that 
$$\int_0^{2\pi } \big\|\mathcal P(j_k(f))(re^{i \theta })\big\|_{\tilde h^{-1}}^t{d\theta \over 2\pi}\leq K \Big({R\over r(R-r)} T_{f, A}(R)\Big)^p. $$ 
The rest of the argument is similar to the proof of Theorem \ref{Brotbekdefect}.  We eventually will get 
	$$
		\int_{B(1)} e^{\tilde u} dV  < \infty,$$
which contradicts  the result of Yau \cite{Yau76} (see also \cite{Karp82}) that
	$$
		\int_{B(1)}  e^{\tilde u}  dV =  \infty.$$
	This completes the proof.
\end{proof}

\end{document}